\documentclass[a4paper]{article}
\usepackage{amsmath, amsthm, amssymb, graphicx, fullpage, microtype}
\usepackage{float}
\usepackage[dvipsnames]{xcolor}
\usepackage{hyperref,pagecolor}

\newtheorem{thm}{Theorem}

\newtheorem{prop}[thm]{Proposition}
\newtheorem{cor}[thm]{Corollary}
\newtheorem{definition}[thm]{Definition}
\newtheorem{problem}[thm]{Problem}
\theoremstyle{definition}
\newtheorem{case}{Case}[thm]

\theoremstyle{remark}
\newtheorem{rem}{Remark}
\newtheorem{clm}{Claim}[thm]
\renewcommand{\theclm}{\arabic{clm}}
\newenvironment{poc}{\begin{proof}[Proof of Claim $\theclm$]}{\end{proof}}

\newcommand{\eps}{\varepsilon}
\newcommand*{\floor}[1]{\lfloor#1\rfloor}
\newcommand*{\ceil}[1]{\lceil#1\rceil}

\newcommand*{\theorem}[1]{Theorem~\ref{#1}}
\newcommand*{\corollary}[1]{Corollary~\ref{#1}}
\newcommand*{\abs}[1]{\lvert#1\rvert}
\newcommand*{\seq}[3][1]{#2_{#1},\ldots,#2_{#3}}
\newcommand*{\hh}[1][h]{\mathcal{\uppercase{#1}}}
\newcommand*{\cod}[1][h]{\delta_2(\hh[#1])}
\newcommand*{\tc}[1][h]{\operatorname{tc}(\hh[#1])}
\newcommand*{\prob}[2][p]{\mathbb{\uppercase{#1}}(#2)}
\newcommand*{\col}[1]{\operatorname{col}(#1)}
\newcommand*{\bfrac}[2]{\genfrac{(}{)}{}{}{#1}{#2}}

\makeatletter
\def\blfootnote{\xdef\@thefnmark{}\@footnotetext}
\makeatother

\title{Forcing large tight components in $3$-graphs%
\blfootnote{\copyright~2018. This manuscript version is made available under the CC-BY-NC-ND 4.0 license \url{http://creativecommons.org/licenses/by-nc-nd/4.0/}.
To appear in the \textit{European Journal of Combinatorics} \textbf{77} (2019), pp 57--67. DOI: \url{https://doi.org/10.1016/j.ejc.2018.11.001}}}

\usepackage{authblk}

\author[1]{Agelos Georgakopoulos}
\author[2]{John Haslegrave}
\affil[1,2]{{Mathematics Institute}\\
	{University of Warwick}\\
	{CV4 7AL, UK}}
\author[3]{Richard Montgomery}
\affil[3]{{School of Mathematics}\\
	{University of Birmingham}\\
	{B15 2TT, UK}}

\begin{document}
\maketitle

\begin{abstract}
Any $n$-vertex $3$-graph with minimum codegree at least $\floor{n/3}$ must have a spanning tight component, but immediately below this threshold it is possible for no tight component to span more than $\ceil{2n/3}$ vertices. Motivated by this observation, we ask which codegree forces a tight component of at least any given size. The corresponding function seems to have infinitely many discontinuities, but we provide upper and lower bounds, which asymptotically converge as the function nears the origin.
\end{abstract}

\section{Introduction}\label{sec:intro}
This paper addresses the extremal question of which minimum codegree forces a tight component containing at least a certain proportion of the vertices of a $3$-uniform hypergraph.

Connectivity in graphs is a simple concept. A vertex pair in a graph is connected if there is a path between them (or, equivalently, if there is a walk between them),
while a graph is connected if every vertex pair is connected. Generalising this concept to $k$-graphs -- hypergraphs where every edge consists of $k$ vertices -- is not straightforward. We first need to consider how to generalise a path. Several generalisations exist depending on the size of the intersection between consecutive edges, but perhaps the most natural and most studied is the \emph{tight path}. A tight path in a $k$-graph is a subgraph with a vertex ordering so that the edges of the path are exactly the sets of $k$ consecutive vertices. Analogously, a tight walk is a subgraph whose edges can be ordered so that consecutive edges share $k-1$ vertices. A tight path is always a tight walk; note, however, that for $k\geq 3$ two vertices in a $k$-graph can be connected by a tight walk without being connected by a tight path.
A key property of connectivity is that the relation of connectivity between vertices is transitive, and thus we consider tight walks when studying connectivity in hypergraphs.

However, we still need to consider \emph{which} subsets are connected. Here, we say two sets in a $k$-graph $\hh$ are connected if $\hh$ contains a tight walk where the first and last edges respectively contain the two given sets (in either order). The \textit{tight components} of $\hh$ are the equivalence classes of edges, where two edges are related if the hypergraph contains a tight walk between them. Equivalently, then, two sets are connected if they are each contained in edges in the same tight component. 

Kahle and Pittel \cite{KP16} considered the question of when all the $(k-1)$-sets of a $k$-graph are connected in this way, under the name \textit{hypergraph connectivity},
in their work studying the closely-related property of cohomological connectivity. A $k$-graph $\hh$ is \textit{cohomologically connected} if the cohomology group
$H^{k-2}(S,\mathbb Z_2)$ vanishes, where $S$ is the $(k-1)$-dimensional simplicial complex generated by the edges of $\hh$ with complete $(k-2)$-skeleton. The threshold for cohomological connectivity of the binomial random $k$-graph was established by Linial and Meshulam \cite{LM06} for $k=3$ and Meshulam and Wallach \cite{MW09} for $k>3$.  It is shown in \cite[Theorem 1.7]{KP16} that a cohomologically connected $k$-graph is hypergraph connected, and as a consequence the thresholds for the two types of connectedness coincide.

From a combinatorial point of view, however, hypergraph connectivity is the more natural notion.
To study it from an extremal perspective, we need to generalise the minimum degree of a graph to a $k$-graph. We work with the \textit{codegree}, that is the number of ways to extend a given $(k-1)$-set into an edge by adding a single vertex. Write $\delta_{k-1}(\hh)$ for the minimum codegree over all $(k-1)$-sets of $\hh$. It is not hard to see (a proof is given in Section~\ref{sec:basic}) that an $n$-vertex $k$-graph $\hh$ is hypergraph connected if $\delta_{k-1}(\hh)>\frac{n-k}2$, and that this is best possible.

In this paper, we will consider the connectivity of individual vertices in $k$-graphs, concentrating on the case $k=3$. Note that all vertices are connected exactly when the $k$-graph contains a tight component covering all the vertices -- that is, a \emph{spanning} tight component. If an $n$-vertex $k$-graph $\hh$ has minimum codegree $\delta_2(\hh)\geq \lfloor n/3\rfloor$, then it has a spanning tight component (see \corollary{one-third}).
Interestingly, as shown in our recent work with Narayanan~\cite{spheres}, this minimum codegree is asymptotically the same minimum codegree required to guarantee that an $n$-vertex $3$-graph contains a spanning triangulation of a sphere (where the faces correspond to edges in the $3$-graph). In fact, this is the asymptotic minimum codegree required to guarantee a spanning triangulation of any fixed surface in an $n$-vertex $3$-graph~\cite{spheres}.
Since all edges in such a triangulation must be in the same tight component, any $3$-graph with a spanning triangulation of a surface contains a spanning tight component.

Another property that immediately gives a spanning tight component is the existence of a spanning tight path, also known as a Hamiltonian tight path. Along with the existence of a spanning tight cycle, this has been well studied (see, for example,~\cite{KK99,RRS06}). In particular, R\"odl, Ruci\'nski and Szemer\'edi~\cite{RRS11} showed that, for sufficiently large $n$, minimum codegree at least $\lfloor n/2\rfloor$ is enough to guarantee a spanning tight cycle. This minimum codegree condition is best possible. Tight paths and cycles covering a constant proportion of the vertices have also been well studied (see, for example,~\cite{ABCM,GKL16}).

If the minimum codegree is not sufficiently high to force a spanning tight component, then how large a tight component is guaranteed?
In the random analogue of this question, beneath the threshold for a spanning tight component there is, similarly to the case with graphs, a giant tight component. The emergence and size of this giant tight component in random hypergraphs was analysed (in more general work) by Cooley, Kang and Person \cite{CKP} and Cooley, Kang and Koch \cite{CKK}.
To address the extremal question, we consider the following function.
\begin{definition}\label{function}
We define a function $f_k(x): [0,1] \to [0,1]$ by letting $f_k(x)$ be the largest real number such that every $n$-vertex $k$-graph with minimum codegree at least $xn-O(1)$ has a tight component meeting at least $f_k(x)n$ of its vertices. (Omitting the ${}-O(1)$ term in this definition changes $f_r$ only slightly: from left-continuous to right-continuous.)
\end{definition}

The function $f_2$ is easy to analyse. Any $n$-vertex graph $G$ with $\delta(G)\geq\floor{n/m}$ can have at most $m-1$ components, so one of them meets at least $\ceil{n/(m-1)}$ vertices. Conversely, if $k<\floor{n/m}$ there is a graph with $m$ components meeting $\floor{n/m}$ or $\ceil{n/m}$ vertices which has minimum degree at least $k$. Thus $f_2(x)=\frac1{\floor{1/x}}$. 

In this paper we analyse $f_3$. \theorem{main} gives our upper and lower bounds, which become asymptotically tight as $x\to 0$.
Our upper bounds are based on the existence of finite projective planes of certain orders. It follows that $f_3$ is discontinuous at $x=1/3$.
We conjecture that, like $f_2$, it has infinitely many discontinuities.

\begin{figure}
\begin{center}
\includegraphics[width=0.5\linewidth]{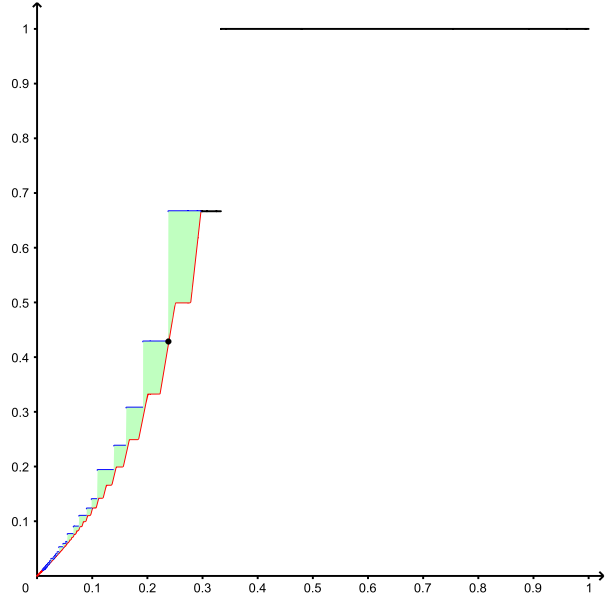}
\end{center}
\caption{Our upper (blue, if colour is shown) and lower (red) bounds on $f_3(x)$.}\label{fig:bounds}\end{figure}

\begin{thm}\label{main}Let $(r_i)_{i\geq 0}$ be the sequence of integers such that $r_i-2$ is a prime power or $0$,
and let $q_i=\frac{r_i-3+\frac2{r_i-1}}{r_i^2-3r_i+3}$.
Then for every $i\geq 0$ we have
\[f_3(x)\leq\frac{r_i-1}{r_i^2-3r_i+3}\quad\text{for }x\in(q_{i+1},q_i]\,.\]
Furthermore,
\[f_3(x)\geq\begin{cases}1&\quad\text{if }x>\frac13\,;\\
\frac23&\quad\text{if }\frac8{27}\leq x\leq\frac13\,;\\
9x-2&\quad\text{if }\frac5{18}\leq x\leq\frac8{27}\,;\\
\frac{1}{r-1}&\quad\text{if }\frac{1}{r+1}\leq x\leq\bfrac{3r-4}{3r-3}\frac 1r\text{ where }r\geq 3\,;\\
\frac{3rx-2}{r-2}&\quad\text{if }\bfrac{3r-4}{3r-3}\frac 1r\leq x\leq\frac 1r\text{ where }r\geq 4\,.
\end{cases}\]
\end{thm}
Figure~\ref{fig:bounds} shows the upper and lower bounds described in \theorem{main}. The bounds coincide for $x\in\bigl\{\frac5{21}\bigr\}\cup\bigl[\frac8{27},1\bigr]$.

A result of similar flavour, both in the problem studied and in the behaviour observed, is obtained by Allen, B\"ottcher and Hladk\'y \cite{ABH11}:
they show that the minimum degree guaranteeing the existence of the square of a path or cycle of a given length follows a step-like pattern similar
to that of Figure~\ref{fig:bounds}.

The higher functions $f_k, k>3$ might be much harder to analyse. In particular, our constructions based on projective planes do not directly
translate to $k$-graphs with large codegree, since they only guarantee that every pair of vertices is contained in many edges.
The following observation shows that our lower bounds do apply to all $k\geq 3$.
\begin{prop}For all $x\in[0,1]$ and all $k\geq 3$ we have $f_k(x)\geq f_{k-1}(x)$.\end{prop}
\begin{proof}For an $n$-vertex $k$-graph $\hh$ and vertex $v$, write $\hh_v$ for the link $(k-1)$-graph on $n-1$ vertices of all $(k-1)$-tuples which form an edge with $v$. Note that if $\hh$ has minimum codegree at least $xn-O(1)$ then so does $\hh_v$, and so $\hh_v$ has a tight component meeting at least $f_{k-1}(x)(n-1)$ vertices. The corresponding edges of $\hh$ are all in the same tight component, which meets at least $f_{k-1}(x)(n-1)+1\geq f_{k-1}(x)n$ vertices.\end{proof}
Our lower bounds on $f_3$ give in particular arbitrarily small values of $x$ for which $f_3(x)\geq\frac{1}{1/x-2}$. This bound is in some sense best possible, since our upper bounds also give arbitrarily small values of $x$ for which $f_3(x)<\frac{1}{1/x-2}$. However, for $k>3$ we might expect a better bound of the form $\frac{1}{1/x-c}$ to hold for small $x$, where $c>2$.
\begin{problem}
Provide asymptotic formulae for $f_k$, when $k>3$.
\end{problem}

\section{Hypergraph connectivity and spanning tight components}\label{sec:basic}
Before starting the analysis of $f_3$ which will be the focus of this paper, we give a justification for the extremal codegree forcing hypergraph connectivity mentioned in Section~\ref{sec:intro}. Recall that a $k$-graph $\hh$ is hypergraph connected if every two $(k-1)$-tuples of vertices of $\hh$ are connected by a tight walk.
\begin{prop}Any $k$-graph $\hh$ on $n\geq k$ vertices with minimum codegree exceeding $\frac{n-k}{2}$ is hypergraph connected, and this is best possible.\end{prop}
\begin{proof}A $k$-graph with minimum codegree $\floor{\frac{n-k}{2}}$ which is not hypergraph connected may be constructed by choosing any set $W$ of $\floor{\frac{n-k}{2}}+1$ vertices, and defining the edges of $\hh$ to be all $k$-tuples which do not meet $W$ in exactly one vertex; no $(k-1)$-tuple meeting $W$ is connected to any $(k-1)$-tuple avoiding $W$.

Suppose $\hh$ has minimum codegree exceeding $\frac{n-k}{2}$, and let $A,B$ be distinct $(k-1)$-tuples. Then either $A\cup x\in E(\hh)$ for some $x\in B\setminus A$ or
$B\cup y\in E(\hh)$ for some $y\in A\setminus B$ or $A\cup z, B\cup z\in E(\hh)$ for some $z\not\in A\cup B$. In each of these cases we may find $(k-1)$-tuples $A', B'$ with more common elements than $A,B$ with the property that $A,B$ are connected if and only if $A',B'$ are connected. Since any $(k-1)$-tuple is connected to itself, it follows that $\hh$ is hypergraph connected.\end{proof}

Now we turn to the corresponding problem for connectivity of vertices. In this section we note that a minimum codegree of $\floor{n/3}$ is sufficient to force a spanning tight component, as shown below. This fact was pointed out to us by Richard Mycroft (private communication). This bound is best possible as proved by the following example. Consider the $3$-graph whose vertices are partitioned into three sets $V_0,V_1,V_2$, of as equal sizes as possible, with all edges consisting of three vertices in $V_i$ or of two vertices in $V_i$ and one in $V_{i+1}$, for some $i\in\mathbb Z_3$. Each tight component only meets vertices in two parts, so is far from spanning, yet the minimum codegree is $\floor{n/3}-1$.

\begin{prop}[R.~Mycroft (private communication)]
Any $3$-graph $\hh$ on $n$ vertices with minimum codegree at least $\floor{n/3}$ has at most two tight components.\end{prop}
\begin{proof}Suppose not. Let  $K_n$ denote the complete graph on the vertices of $\hh$. Colour the hyperedges of $\hh$ according to the tight component they are in, and give each edge $e=uv$ of $K_n$ the colour of those hyperedges of $\hh$ which contain $\{u,v\}$.

First, we show that no triangle of $K_n$ has more than two colours. If $xy$, $yz$ and $zx$ are different colours then consider the sets of vertices $A,B,C$ which can be used to extend $xy,yz,zx$ respectively to hyperedges. In the colouring of $K_n$, every vertex in $A$ has two edges of the first colour to $\{x,y,z\}$, etc., so these sets are disjoint from each other and $\{x,y,z\}$. But each has size at least $\floor{n/3}$, so we have $n\geq 3+3\floor{n/3}$, a contradiction.

Second, we show that no vertex meets three colours. If $v$ does, say red, green and blue, let $R$ (respectively, $G$ or $B$) be the sets of vertices adjacent to $v$ by red (respectively, green or blue) edges. These are disjoint, but if $vx$ is any red edge then there are at least $\floor{n/3}$ vertices which extend it to a red hyperedge, so $\abs{R}\geq 1+\floor{n/3}$. The same applies to $G$ and $B$, so $\abs{R}+\abs{G}+\abs{B}>n$, contradiction.

Now consider any vertex $v$ which meets edges of two colours (this trivially exists), say red and blue. Let $R$ be the set of all vertices with red edges to $v$, and define $B$ similarly. $R$ and $B$ partition $V-v$. By assumption, a third colour, green, is used somewhere; it cannot be between $R$ and $B$ as there are no $3$-coloured triangles, so it is within $R$, say. If $xy$ is such an edge then $vx$ extends to $\floor{n/3}$ red hyperedges and $xy$ extends to $\floor{n/3}$ green hyperedges. Each of the vertices which extends one of these two is in $R$ (if $vxz$ is a red hyperedge then $vz$ is red; if $xyz$ is green then $z\not\in B$ since that would create a $3$-coloured triangle). So $\abs{R}\geq 2+2\floor{n/3}$ and as before $\abs{B}\geq 1+\floor{n/3}$, giving a contradiction.
\end{proof}

\begin{cor}[R.~Mycroft (private communication)]\label{one-third}
Any $3$-graph $\hh$ on $n$ vertices with $\cod\geq\floor{n/3}$ has a spanning tight component.\end{cor}
\begin{proof}If the first tight component does not meet some vertex $x$ then for each other vertex $y$, the edges containing $x$ and $y$ must belong to the other tight component, which therefore meets all vertices.\end{proof}

The example given above shows that reducing the minimum codegree condition even by $1$ allows hypergraphs where no tight component meets more than $\ceil{2n/3}$ vertices. This is the motivation for Definition~\ref{function}: we have shown that $f_3(x)=1$ for all $x>1/3$, but $f_3(1/3)\leq 2/3$.

We will show that a minimum codegree of $n/r-O(1)$ implies that some tight component meets at least $n/(r-2)-O(1)$ vertices for each integer $r\geq 3$. We also show that this is almost best possible: for infinitely many values of $r$ there are hypergraphs with minimum codegree $(1/r-O(r^{-3}))n$ in which every tight component meets fewer than $n/(r-2)$ vertices.

\section{Upper bounds}\label{sec:ub}
In this section we give a construction based on finite projective planes. (As we only consider finite projective planes in this paper, we shall henceforth omit to specify finiteness.) A projective plane of order $s$ is an arrangement of points and lines such that each point lies on $s+1$ lines, each line contains $s+1$ points, each pair of points is contained in a unique line, and each pair of lines meet in a unique point. Such a structure is known to exist whenever $s$ is a prime power. Bruck and Ryser \cite{BR49} proved that if $s\equiv 1$ (mod $4$) or $s\equiv 2$ (mod $4$), and $s$ is not the sum of two squares, then no projective plane exists. The existence of a projective plane of order $10$ was ruled out by extensive computer analysis, completed by Lam, Thiel and Swiercz \cite{LTS89}, but for every other value of $s$ which is neither a prime power nor ruled out by the Bruck--Ryser result, it is an open question. We will consider a projective plane as a hypergraph, where the vertices are the points and the edges are the lines.

Let $\tc$ denote the number of vertices of the largest tight component of the $3$-graph $\hh$.

\begin{thm}\label{construct}For each $r\geq 3$ for which a projective plane of order $r-2$ exists, and any $n$, there exists an $n$-vertex $3$-graph $\hh$ satisfying
\begin{gather*}
\cod=\bfrac{r-3+\frac2{r-1}}{r^2-3r+3}n-O(1)\quad\text{and}\\
\tc=\bfrac{r-1}{r^2-3r+3}n+O(1)\,.
\end{gather*}\end{thm}
\begin{rem}In fact, provided $r^2-3r+3\mid n$, we do not need the ${}+O(1)$ term in the latter expression.\end{rem}
\begin{rem}Writing $x=\frac{r-3+\frac2{r-1}}{r^2-3r+3}$, we have $x=1/r-O(r^{-3})$ and
\[f_3(x)\leq\frac{r-1}{r^2-3r+3}<\frac{1}{1/x-2}<\frac{1}{r-2}\,.\]\end{rem}
\begin{proof}
Let $\hh[P]_{r-2}$ be a projective plane of order $r-2$; this is an $(r-1)$-uniform hypergraph with $r^2-3r+3$ vertices $\seq v{r^2-3r+3}$ and $r^2-3r+3$ edges, with each vertex having degree $r-1$ and each pair of vertices contained in exactly one edge. Associate each edge with a different colour.

Colour the complete graph $K_n$ on $n$ vertices as follows. Divide the vertices as evenly as possible into $r^2-3r+3$ classes $\seq C{r^2-3r+3}$.
For each class $C_i$, using the $r-1$ colours corresponding to the edges of $\hh[p]_{r-2}$ meeting $v_i$, colour the edges within $C_i$ such that
for each vertex the numbers of incident edges of each colour are as equal as possible.
Note that at any vertex the discrepancy between any two colours will be at most $2$, since this can be achieved by partitioning the edges within $C_i$
into matchings of size $\floor{\abs{C_i}/2}$ and making each matching monochromatic.
Colour each edge between two classes $C_i$ and $C_j$, where $i\neq j$, according to the unique edge of $\hh[p]_{r-2}$ which contains $v_i$ and $v_j$. Figure~\ref{fig:proj} shows such a colouring for $r=4$.

\begin{figure}\begin{center}
\includegraphics[width=.75\textwidth]{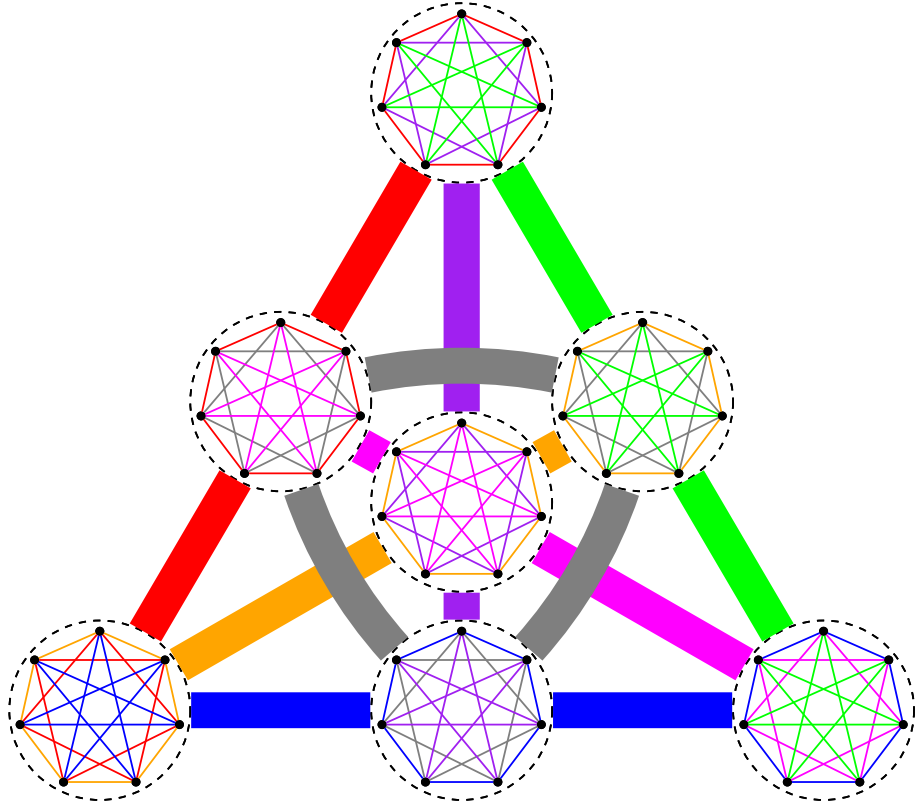}
\end{center}\caption{A construction based on $\hh[p]_2$.}\label{fig:proj}\end{figure}

Now define a hypergraph $\hh$ on the vertex set of $K_n$, whose edges are the monochromatic triangles of this colouring of $K_n$. If we give edges of $\hh$ the same colour as the corresponding triangle, each tight component is monochromatic, and each colour touches $r-1$ classes $C_i$, so each tight component has at most as many vertices as the $r-1$ largest classes.

Fix a colour $c$. Provided $n$ is sufficiently large in terms of $r$, we may choose classes $C_i$ and $C_j$ and vertices $v,w\in C_i$ and $x,y\in C_j$ such that $vwx$ and $wxy$ are hyperedges of colour $c$. For any other hyperedge of colour $c$, there is a tight path of length at most two to an edge containing $vw$ or $xy$, and so each colour corresponds to exactly one tight component. Thus we have
\[\tc=\frac{r-1}{r^2-3r+3}n+O(1)\,.\]
Fix a pair of vertices $x,y\in V(\hh)$, and let $e$ be the edge of $\hh[p]_{r-2}$ corresponding to the colour of $xy$. The {\em degree} $d_{\hh}(x,y)$ is the number of hyperedges of $\hh$ containing both $x,y$. If $x$ and $y$ are in the same class, all vertices in the other $r-2$ classes corresponding to vertices of $e$ form monochromatic triangles with $x,y$, so we have
\[d_{\hh}(x,y)\geq\frac{r-2}{r^2-3r+3}n+O(1)\,.\]
If $x\in C_i$ and $y\in C_j$ with $i\neq j$, then all vertices in the other $r-3$ classes corresponding to vertices of $e$ form monochromatic triangles with $x,y$, as do the vertices in $C_i$ with an appropriately coloured edge to $x$, and those in $C_j$ with an appropriately coloured edge to $y$. In total, we have
\begin{equation}d_{\hh}(x,y)=\frac{r-3+\frac2{r-1}}{r^2-3r+3}n+O(1)\,.\label{equality}\end{equation}
Thus we have
\[\cod\geq\min\biggl(\frac{r-2}{r^2-3r+3},\frac{r-3+\frac2{r-1}}{r^2-3r+3}\biggr)n+O(1)\,.\]
For $r=3$ the two bounds coincide, and for $r>3$ the latter is smaller. Since there are always some pairs satisfying \eqref{equality}, we have the required equality.\end{proof}
Recall that $(r_i)_{i\geq 0}$ is the sequence of integers such that $r_i-2$ is a prime power or $0$,
i.e.\ the sequence that begins $2,3,4,5,6,7,9,10,11,13,15,\ldots.$ For each $i\geq 1$, \theorem{construct} shows that
\[f_3\bfrac{r_i-3+\frac2{r_i-1}}{r_i^2-3r_i+3}\leq\frac{r_i-1}{r_i^2-3r_i+3}\,\]
and the fact that $f_3(x)$ is increasing, together with the trivial bound $f_3(x)\leq 1$, gives the upper bounds claimed in \theorem{main}
\section{Lower bounds}
Next we give a lower bound which is close to the upper bound of the previous section for large $r$. We will use the following result of F\"uredi on fractional matchings in hypergraphs \cite{Fur81}. A \textit{matching} in a hypergraph $\hh$ is a set of disjoint edges, and the matching number $\nu(\hh)$ is the maximum size of a matching in $\hh$. A \textit{fractional matching} is a weight function $w:E(\hh)\to[0,1]$ such that $\sum_{e\ni v}w(e)\leq 1$ for each $v\in V(\hh)$, and the \textit{fractional matching number} $\nu^*(\hh)$ is the maximum of $\sum_{e\in E(\hh)}w(e)$ over all fractional matchings.

\begin{thm}[F\"uredi \cite{Fur81}]\label{furedi}Let $\hh$ be a hypergraph with edges of size at most $k$ which does not contain $p+1$ vertex-disjoint projective planes of order $k-1$, for some $k\geq 3$ and $p\geq 0$. Then $\nu^*(\hh)\leq(k-1)\nu(\hh)+p/k$.\end{thm}
We write $\Delta_1(\hh)$ for the maximum vertex degree of a hypergraph $\hh$. A hypergraph is \textit{intersecting} if any two edges intersect.
\begin{cor}\label{intersecting}If $\hh$ is an intersecting $k$-uniform multi-hypergraph then
\[\Delta_1(\hh)\geq\frac{e(\hh)}{k-1+p/k}\,,\]
where $p=1$ if a projective plane of order $k-1$ exists, and $p=0$ otherwise. Further, if $k\geq 3$ and $\Delta_1(\hh)<\frac{e(\hh)}{k-1}$
then the underlying simple hypergraph is a projective plane.\end{cor}
\begin{proof}If $k=2$ then the claimed bound is $\Delta_1(\hh)\geq\frac{e(\hh)}{1+1/2}=2e(\hh)/3$ and either there is a vertex in every edge
(giving $\Delta_1(\hh)=e(\hh)$) or $\hh$ has only three vertices; in the latter case the average degree is $2e(\hh)/3$,
giving $\Delta_1(\hh)\geq2e(\hh)/3$ as required.

If $k\geq 3$ we may apply \theorem{furedi} to the underlying simple hypergraph $\hh'$. Since $\hh'$ is intersecting,
$\nu(\hh')=1$, and so $\nu^*(\hh')\leq k-1+p/k$. For each $e\in E(\hh')$, let $w(e)$ be the number of copies of $e$ in the multi-hypergraph $\hh$, divided by $\Delta_1(\hh)$.
Clearly, for each $v\in V(\hh')$,
\[\sum_{e\ni v}w(e)=\frac{d_{\hh}(v)}{\Delta_1(\hh)}\leq 1\,,\]
so $w$ is a fractional matching for $\hh'$. Thus,
\[k-1+p/k\geq\sum_{e\in E(\hh')}w(e)=\frac{e(\hh)}{\Delta_1(\hh)}\,,\]
giving the required bound. If $k\geq 3$ and $\Delta_1(\hh)<\frac{e(\hh)}{k-1}$ then $\nu^*(\hh')\geq\frac{e(\hh)}{\Delta_1(\hh)}>k-1$ so $\hh'$ contains a projective plane, and, since it is intersecting, no other edges.\end{proof}
\begin{thm}\label{lwrbnd}Fix an integer $r\geq 3$. Suppose $\hh$ is a $3$-uniform hypergraph on $n$ vertices with $\cod\geq(1-\eps)n/r$,
where $0\leq\eps<\frac 1{r+1}$. Then
\begin{equation}\label{bounds}\tc\geq\begin{cases}\min\{(1-3\eps),2/3\}n\quad &\text{if }r=3\\(1-3\eps)\dfrac{n}{r-2}\quad &\text{otherwise.}\end{cases}\end{equation}
\end{thm}
\begin{proof}
Again, we colour the edges of the complete graph on the same vertex set. Give $xy$ the colour of the tight component containing edges of the form $xyz$. Fix a vertex $x$. If $x$ meets an edge $xy$ of a particular colour in the graph, there are at least $\cod$ edges of the form $xyz$ in $\hh$ which are in the corresponding tight component. Thus if $x$ meets an edge of a certain colour, it meets at least $\cod$ such edges. Since $\eps<1/(r+1)$, $\cod>n/(r+1)$, so the number of tight components meeting a vertex $x$ is at most $r$. We distinguish three cases, as follows.
\begin{case}Some vertex meets at most $r-2$ tight components.\end{case}
In this case, these $r-2$ components must between them cover all the vertices, so at least one must meet at least $n/(r-2)$ vertices.
\begin{case}Every vertex meets exactly $r-1$ tight components.\end{case}
We define an auxiliary multi-hypergraph $\hh[f]$ as follows. The vertices of $\hh[f]$ correspond to tight components of $\hh$. The edges of $\hh[f]$ correspond to vertices of $\hh$; an edge $e_v$ of $\hh[f]$ corresponding to a vertex $v$ of $\hh$ contains the $r-1$ vertices of $\hh[f]$ corresponding to tight components which meet $v$. Thus $\hh[f]$ is $(r-1)$-uniform and $e(\hh[f])=n$.
Any two edges of $\hh[f]$ intersect: $e_u$ and $e_v$ both contain the vertex corresponding to the tight component containing edges of the form $uvw$. If $r=3$, by \corollary{intersecting}, such an $\hh[f]$ has a vertex meeting at least $2n/3$ edges, and hence $\hh$ has a tight component meeting this many vertices. If $r>3$, by \corollary{intersecting}, either $\hh[f]$ has a vertex meeting at least $n/(r-2)$ edges or its underlying simple hypergraph $\hh[f]'$ is a projective plane of order $r-2$. In the former case we are done; in the latter case some vertex of $\hh[f]$ meets at least $\frac{r-1}{r^2-3r+3}n=\frac{1}{r-2+1/(r-1)}n$ edges. Note that
\[\frac{1}{r-2+1/(r-1)}=\biggl(1-\frac{1}{r^2-3r+3}\biggr)\frac{1}{r-2}\,.\]
\begin{clm}If $\hh[f]'$ is a projective plane of order $r-2$ then $\eps\geq\frac{r-3}{(r-1)(r^2-3r+3)}$.\end{clm}
\begin{poc}Colour each pair of vertices of $\hh$ according to the tight component the edges containing that pair are in.

If $\hh[f]'$ is a projective plane of order $r-2$, partitioning the edges of $\hh[f]$ according to which edge of $\hh[f]'$ they correspond to
gives a partition, $C_1,\ldots,C_{r^2-3r+3}$ say, of the vertices of $\hh$ so that any two vertices in the same class are in exactly the same tight components.
Each class meets $r-1$ tight components, and each pair of classes have a single tight component in common. Thus we may define a hypergraph $\hh[f]''$ having
one vertex $w_i$ for each class $C_i$ and one edge for each tight component, containing the vertices corresponding to classes it meets;
$\hh[f]''$ is also a projective plane of order $r-2$ (dual to $\hh[f]'$).

We fix a vertex $z$ and then choose a pair of vertices $(x,y)$ as follows: choose uniformly at random between the ordered pairs $(i,j)\in[r^2-3r+3]^2$ which satisfy $i\neq j$, and choose (independently and uniformly at random) $x\in C_i$ and $y\in C_j$. Now consider $\prob{xyz\in E(\hh)}$. If $z\in C_k$ we have
\begin{align}\prob{xyz\in E(\hh)}&=\prob{(xyz\in E(\hh))\wedge (k\not\in \{i,j\})}+\prob{(xyz\in E(\hh))\wedge (k\in\{i,j\})}\nonumber\\
&\leq\prob{w_iw_jw_k\in E(\hh[f]'')}+2\prob{(i=k)\wedge(\col{xy}=\col{xz})}\label{ineq1}\,.
\end{align}
Now
\begin{align}\prob{w_iw_jw_k\in E(\hh[f]'')}&=\frac{(r-1)\binom{r-2}2}{\binom{r^2-3r+3}2}\nonumber\\
&=\frac{r-3}{r^2-3r+3}\label{ineq2}\,,
\end{align}
and
\begin{align}
\prob{(k=i)\wedge(\col{xy}=\col{xz})}&<\frac{1}{r^2-3r+3}\cdot\frac{r-2}{r^2-3r+2}\nonumber\\
&=\frac{1}{(r-1)(r^2-3r+3)}\,,\label{ineq3}
\end{align}
since given $i=k$, $\col{xy}$ depends only on $j$, and for each choice of $x$ (other than $x=z$) there are $r-2$ choices of $j$ which give $\col{xy}=\col{xz}$.

Thus, by \eqref{ineq1}, \eqref{ineq2} and \eqref{ineq3}, we have
\begin{align*}\cod&\leq\prob[e]{d_{\hh}(x,y)}\\
&<\frac{r-3+\frac{2}{r-1}}{r^2-3r+3}n\\
&=\biggl(1-\frac{r-3}{(r-1)(r^2-3r+3)}\biggr)\frac nr\,,
\end{align*}
as required. This completes the proof of Claim~\theclm.
\end{poc}
The desired result follows in this case, since if $\eps\geq\frac{r-3}{(r-1)(r^2-3r+3)}$ and $r\geq 4$ then $3\eps\geq \frac{1}{r^2-3r+3}$.
\begin{case}Neither of the above two cases apply.\end{case}
In this case, some vertex $x$ meets exactly $r$ tight components. Divide the remaining $n-1$ vertices into classes $\seq Ar$ of sizes $\seq ar$, according to which tight component edges containing a given vertex and $x$ are in. Note that $a_i\geq \cod \geq(1-\eps)n/r$ for each $i\in [r]$. If these are the only tight components then, since each vertex is met by at least $r-1$ of them, some component meets at least $n(r-1)/r$ vertices. So we may assume that there is another tight component $\mathcal B$ which does not meet $x$. Suppose $\mathcal B$ meets $b_i$ vertices in $A_i$ for each $i$, and in total meets $b$ vertices. Write $S=\{i\in[r]:b_i>0\}$.

If $\abs{S}\leq 2$ then the codegree of any pair which is in an edge of $\mathcal B$ and meets all the (at most 2) parts, is at most $\sum_{i\in S} (a_i-\cod)\leq \sum_{i\in {[r]}} (a_i-\cod) \leq n\eps<n/(r+1)<\cod$, giving a contradiction. So we may assume $\abs{S}\geq 3$.
\begin{clm}For each pair $i\neq j\in S$,
\begin{equation}3\cod\leq b-b_i-b_j+a_i+a_j\,.\label{codeg}\end{equation}\end{clm}
\begin{poc}
First, suppose that there are vertices $v_i\in A_i$ and $v_j\in A_j$ such that edges containing $v_i,v_j$ are in $\mathcal B$. Consider the codegree of this pair.
For each $k\neq i,j$, the only vertices in $A_k$ which can form an edge with $v_i,v_j$ are the $b_k$ vertices which meet $\mathcal B$.
Also, from $A_i$ only the vertices $w\in A_i$ for which there is no edge of the form $xwv_i$ can form an edge with $v_i,v_j$. Since $d_{\hh}(x,v_i)\geq\cod$, and all vertices which complete an edge with $x,v_i$ lie in $A_i$, at most $a_i-\cod$ such $w$ exist (and similarly for $A_j$). Thus we have
\[d_{\hh}(v_i,v_j)\leq\sum_{k\in S\setminus\{i,j\}}b_k+(a_i-\cod)+(a_j-\cod)\,;\]
rearranging, and noting that $\sum_{k\in S} b_k=b$, gives the desired inequality.

Second, suppose that no such vertices exist. Define $S_i$ to be the set of $k\neq i,j$ such that there exist vertices $v_i\in A_i$ and $v_k\in A_k$ which can be extended to an edge of $\mathcal B$. Note that $S_i$ is non-empty, since otherwise we would have $a_i>2\cod$, giving $\sum_{k\in[r]}a_k>(r+1)\cod$, a contradiction.
For each $k\in S$, consider the vertices which extend the pair $v_i,v_k$ to an edge. No vertex from $A_j$ can do this, and so similar reasoning gives $3\cod\leq b-b_i-b_j-b_k+a_i+a_k$. If none of these gives the desired inequality, we must have $a_k-a_j>b_k$ for each $k\in S_i$.
For any $l\in S_i$, picking $v_i \in A_i$ and $v_l \in A_l$ such that the pair $v_i,v_l$ lies in an edge of $\mathcal{B}$, the pair $v_i,v_l$ must have codegree at most
\begin{align*}
\sum_{k\in S_i\setminus\{l\}}b_k+(a_i-\cod)+(a_l-\cod)&<\sum_{k\in S_i\setminus\{l\}}(a_k-a_j)+(a_i-\cod)+(a_l-\cod)\\
&\leq\sum_{k\in S_i\cup\{i\}}\biggl(a_k-(1-\eps)\frac nr\biggr)\\
&\leq\sum_{k\in [r]}\biggl(a_k-(1-\eps)\frac nr\biggr)\\
&=n\eps-1<n/(r+1)<\cod\,,\end{align*}
a contradiction. This completes the proof of Claim~\theclm.\end{poc}

Now, averaging inequality \eqref{codeg} over all pairs $i\neq j\in S$ we get
\[3\cod\leq b-\frac{2(|S|-1)b}{|S|(|S|-1)}+\frac{2(n-(r-|S|)\cod)}{|S|}
=\frac{r-s-2}{r-s}b+\frac{2n-2s\cod}{r-s}\,,\]
where $s=r-\abs{S}$. Rearranging this (noting that $r-s-2>0$) gives
\begin{align*}b&\geq\frac{(3r-s)\cod-2n}{r-s-2}\\
&\geq\frac{(3r-s)(1-\eps)n-2nr}{r(r-s-2)}\,.\end{align*}
It suffices to show that
\[\frac{(3r-s)(1-\eps)n-2nr}{r(r-s-2)}\geq\frac{(1-3\eps)n}{r-2}\,,\]
or equivalently
\[(r-2)\bigl((3r-s)(1-\eps)-2r\bigr)\geq r(r-s-2)(1-3\eps)\,.\]
But
\[(r-2)\bigl((3r-s)(1-\eps)-2r\bigr)-r(r-s-2)(1-3\eps)=2s(1-(r+1)\eps)>0\,,\]
as required.
\end{proof}
We are now in a position to complete the proof of \theorem{main}; recall that the upper bounds were proved in Section~\ref{sec:ub}.

We stated \theorem{lwrbnd} for the range of $\eps$ for which the proof works. However, in this range we also have $\delta_2(\hh)\geq(1-\eps)\frac nr>\frac n{r+1}$ and so, by \theorem{lwrbnd} for $r+1$, $\tc\geq\frac n{r-1}$. This gives a better bound than \eqref{bounds} for $\eps>\frac1{3r-3}$, leading to the lower bounds described in \theorem{main}.

\section*{Acknowledgements}

The first and second authors were supported by the European Research Council (ERC) under the European Union's Horizon 2020 research and innovation programme (grant agreement no.\ 639046).

\end{document}